\documentclass[12pt]{amsart}
\usepackage[active]{srcltx}
\usepackage{a4wide}
\usepackage{amsthm,amsfonts,amsmath,mathrsfs,amssymb}
\usepackage{dsfont}
\usepackage[T1]{fontenc}
\usepackage[utf8]{inputenc}
\usepackage{fourier}
\usepackage{pdfpages}
\usepackage{graphicx}

\def\a{\alpha}               \def\g{\gamma}
\def\d{\delta}       \def\la{\lambda}     \def\om{\omega}
       \def\t{\theta}       \def\f{\phi}
         \def\r{\rho}         \def\z{\zeta}
\def\ch{\chi}        \def\o{\omega}

\def\e{\varepsilon}       

       \def\De{{\Delta}}    \def\O{\Omega}

\def\D{{\mathbb D}}     \def\T{{\mathbb T}}
\def\C{{\mathbb C}}     \def\N{{\mathbb N}}

\def\ch{{\mathcal H}}   
\def\cl{{\mathcal L}}   
\def\cp{{\mathcal P}}   

\def\({\left(}       \def\){\right)}

\DeclareMathOperator{\esssup}{ess\,sup}
\newtheorem{prop}{\sc Proposition}

\newtheorem{thm}[prop]{\sc Theorem}
\newtheorem{cor}[prop]{\sc Corollary}
\newtheorem{ex}[prop]{\sc Example}

\begin{document}
\title[Weighted compositions preserving the Carath\'eodory class]{On weighted compositions preserving the Carath\'eodory class}
\author[I. Ar\'evalo]{Irina Ar\'evalo}
\address{Departamento de Matem\'aticas, Universidad Aut\'onoma de
Madrid, 28049 Madrid, Spain}
\email{irina.arevalo@uam.es}
\author[R. Hern\'andez]{Rodrigo Hern\'andez}
\address{Universidad Adolfo Ib\'a\~nez, Facultad de Ingenier\'{\i}a y Ciencias, Av. Padre Hurtado 750, Vi\~na del Mar, Chile}
\email{rodrigo.hernandez@uai.cl}
\author[M.J. Mart\'{\i}n]{Mar\'{\i}a J. Mart\'{\i}n}
\address{University of Eastern Finland, Department of Physics and Mathematics, P.O. Box 111, 80101 Joensuu, Finland}
\email{maria.martin@uef.fi}
\author[D. Vukoti\'c]{Dragan Vukoti\'c}
\address{Departamento de Matem\'aticas, Universidad Aut\'onoma de
Madrid, 28049 Madrid, Spain} \email{dragan.vukotic@uam.es}
\thanks{Arévalo, Mart\'in, and Vukoti\'c are supported by MTM2015-65792-P from MINECO and FEDER/EU and partially by the Thematic Research Network MTM2015-69323-REDT, MINECO, Spain. Hernández and Martín are supported by FONDECYT 1150284, Chile. Martín is also supported by Academy of Finland Grant 268009.}
\subjclass[2010]{30C45, 47B33}
\date{17 February, 2017.}
\begin{abstract}
We characterize in various ways the weighted composition transformations which preserve the class $\cp$ of normalized analytic functions in the disk with positive real part. We analyze the meaning of the criteria obtained for various special cases of symbols and identify the fixed points of such transformations.
\end{abstract}
\maketitle
\section{Introduction}
 \label{sec-intro}

\subsection{The class $\cp$ and its properties}
 \label{subsec-class-p}
Let $\D$ denote the unit disk in the complex plane and $\ch(\D)$ the algebra of
all functions analytic in $\D$. If $\f\in\ch(\D)$ and $\f(\D)\subset \D$, we will say that $\f$ is an \textit{analytic self-map\/} of the disk. If, moreover, such $\f$ satisfies $\f(0)=0$, we will refer to it as a \textit{Schwarz-type function\/} (as in the classical Schwarz lemma).
\par
Denote by $\cp$ the Carath\'eodory class of all $f$ in $\ch(\D)$ with positive
real part and normalized so that $f(0)=1$. An important example of a function in this class is the so-called \textit{half-plane mapping\/} $\ell$ given by
$$
 \ell(z)=\frac{1+z}{1-z}\,, \qquad z\in\D\,.
$$
This conformal map of the disk onto the right half-plane is extremal in many senses for the class $\cp$. This is manifested, for example, by the growth theorem for the functions in the class:
$$
  \ell(-|z|) = \frac{1-|z|}{1+|z|} \le |f(z)|\le \frac{1+|z|}{1-|z|} = \ell(|z|) \,.
$$
The above estimate \cite[Section~2.1]{Pb} is a direct corollary of the Schwarz lemma and the elementary subordination principle since every function $f$ in $\cp$ is of the form $\ell\circ\om$, where $\om$ is some Schwarz-type function. In the particular case when $\om(z)=\lambda z$ with $|\lambda|=1$, we will use the symbol $\ell_\lambda$ to denote the functions $\ell\circ \om$; that is, $\ell_\la (z) = (1+\la z)/(1-\la z)$. In view of the Herglotz representation theorem \cite[Chapter~1]{D2}, the class $\cp$ equals  $\overline{\textrm{co\,} (\cl)}$, the closed convex hull of the collection  $\cl= \{\ell_\la\, \colon\,|\la|=1\}$ in the topology of uniform convergence on compact subsets of $\D$ (the compact-open topology). Every Schwarz-type function has radial limits almost everywhere on the unit circle $\T$ with respect to the normalized arc length measure $dm=d\t/(2\pi)$ (see \cite[Chapter~1]{D1}), hence so does every $f$ in $\cp$.

\subsection{On weighted composition transformations}
 \label{subsec-wct}
Whenever $\f$ is an analytic self-map of the disk, it is immediate that $f\circ\f\in\cp$ for every $f\in \cp$ if and only if $\f$ is a Schwarz-type function. Thus, it makes sense to consider the \textit{composition transformation\/} $C_\f$ on $\cp$ defined by the formula $C_\f(f)=f\circ\f$. For the theory of composition operators on Banach or Hilbert spaces of analytic functions, we refer the reader to \cite{CM} or \cite{S}.
\par
It also seems reasonable to consider the \textit{multiplication  transformations\/} on $\cp$, given by $M_F(f)=F f$, where $F\in\ch(\D)$. While there are many cases of such transformations in spaces of analytic functions, it turns out that such mapping can preserve the class $\cp$ only in the trivial case when $F\equiv 1$.
\par
One can consider the more general \textit{weighted composition transformations\/} $T_{F,\f}$ on $\cp$, defined by a composition followed by a multiplication: $T_{F,\f}(f)=F (f\circ\f)$, where $F$, $\f\in\ch(\D)$ and $\f(\D)\subset\D$, and ask whether such a transformation \textit{acts\/} on $\cp$ or \textit{preserves\/} $\cp$ in the sense that $T_{F,\f}(f) \in\cp$ for every $f$ in $\cp$. This will make sense more often since one of the two symbols $F$ and $\f$ can compensate for the behavior of the other. Also, the nonlinearity of such transformations in our context makes the analysis non-trivial. \par
The linear weighted composition operators have been studied in a number of papers in the context of Banach spaces of analytic functions. The earliest references on this type of operators are \cite{Ka} and \cite{Ki}; among the numerous more recent references, we mention \cite{AL}, \cite{CH}, or \cite{CZ}. Although the class $\cp$ has no linear space structure, studying the question of when a weighted composition transformation preserves this class should be of interest as it would indicate the degree of rigidity when trying to produce new functions with positive real part from the given ones by composing and multiplying. Moreover, the knowledge of weighted composition transformations that preserve the class $\cp$ could even have some applications to variational methods for solving non-linear extremal problems in geometric function theory.
\par
\subsection{Summary of main results}
 \label{subsec-descr-main-results}
The question considered in the present paper is to characterize the  ordered pairs of functions $(F,\f)$ for which $F (f\circ\f)\in\cp$ whenever $f\in\cp$. We solve this problem completely by exhibiting explicit conditions of geometric and analytic character which are equivalent to this property and are relatively easy  to check in practice. The precise statement is formulated as Theorem~\ref{thm-main}.
\par
Although this result is conclusive, the work does not end here by any means. In some situations the theorem allows us to deduce the following rigidity principle: if $\f$ is of certain type, then the only $F$ possible in the trivial case: $F\equiv 1$; see Proposition~\ref{prop-fi-inner}. Alternatively, when $F$ is of special type then $\f$ must be unique; \textit{cf\/.} Proposition~\ref{prop-om-inner}.
\par
Another issue of interest is to construct non-trivial examples and to interpret the statement of Theorem~\ref{thm-main} for some specific types of maps $F$ and $\f$ in terms of their qualitative behavior (image compactly contained in $\D$, radial limits, angular derivatives, etc.). Roughly speaking, one could ask how much we can push the conditions for these maps to the limit and when one of the symbols is ``good'' in some sense, how ``bad'' can the other one be. Thus, various statements such as propositions~\ref{prop-comp-cont} and \ref{prop-comp-cont-dual} and Theorem~\ref{thm-ang-deriv} seem to require a more intricate analysis and may be as interesting as the characterization itself. Different relevant examples are also included.
\par
We end the paper by showing that every transformation $T_{F,\f}$ that preserves class $\cp$ has a unique fixed point whenever $\f$ is not a rotation and this fixed point is obtained by applying the iterates of $T_{F,\f}$ to an arbitrary function in $\cp$, as one would expect. This is the content of Theorem~\ref{thm-fixed-pts}. In the case when $\f$ is a rotation, one can easily describe all fixed points. Even though the transformation is non-linear, here one can adapt the methods from linear functional analysis used in the study of composition operators.

\section{Characterizations of the $\cp$-preserving weighted compositions}
 \label{sect-wco-p}

\subsection{Initial observations}
 \label{subsec-init-obs}
We begin by recording the obvious necessary conditions that must be satisfied by all \textit{admissible symbols\/} $F$ and $\f$, \textit{i.e.}, by those which satisfy the initial assumptions $F$, $\f\in\ch(\D)$, $\f(\D)\subset\D$ and make the inclusion $T_{F,\f}(\cp)\subset\cp$ possible.
\par
$\bullet$ If $F (f\circ\f)\in\cp$ for all $f$ in $\cp$, it is clear that then $F\in\cp$; this is immediately seen by choosing $f\equiv 1$, which is a function in $\cp$.
\par
$\bullet$ In addition to this, $\f$ must be a Schwarz-type function since, after choosing $f(z)=1+z$, another function obviously in $\cp$, we get
$$
 1 = F(0) f(\f(0)) = 1+ \f(0)\,,
$$
hence $\f(0)=0$.
\par
Thus, from now on we shall always work assuming these hypotheses: $F\in\cp$ and $\f$ is a Schwarz-type function.
\par\smallskip
It is quite easy to establish the lack of non-trivial pointwise multipliers of $\cp$.
\begin{prop} \label{prop-no-mult}
If $T_{F,\f}$ preserves $\cp$ and $\f(z)=z$ (that is, $F\,f\in \cp$ for all $f$ in $\cp$) then $F\equiv 1$.
\end{prop}
\begin{proof}
Indeed, since then $F f\in\cp$, by the growth theorem for the functions in $\cp$ we have
\[
 |F(z)|\cdot |f(z)| \leq \frac{1+|z|}{1-|z|}
\]
for all $z$ in $\D$. Also, for any fixed $z$ we may choose $f$ to be a suitable rotation of the half-plane function for which
\[
 |f(z)| = \frac{1+|z|}{1-|z|}\,.
\]
It follows that $|F(z)|\le 1$ for all $z\in\D$. Since $F(0)=1$, the maximum modulus principle implies that $F$ is identically constant, hence $F\equiv 1$.
\end{proof}

\subsection{The main theorem}
 \label{subsec-main-res}
We now characterize in different ways all admissible pairs of symbols. Note that condition (b) below simply states that it suffices to test the action of $T_{F,\f}$ on the set $\cl=\{\ell_\la\,\colon\,|\la|=1\}$ of all rotations of the half-plane function in order to know whether the transformation preserves $\cp$. Condition (c) gives an effective analytic way of testing if a symbol is  admissible or not while (d) provides conditions of geometric type. Each can be useful in its own way.
\par
Throughout the paper, we shall consider the principal branch of the argument function with  values in $(-\pi,\pi]$. Note that for any function $f$ in $\cp$, the function $\arg f$ takes on the values only in $(-\pi/2,\pi/2)$ and is a continuous function in the disk. Moreover, the argument of the product of two such functions $f$ and $g$ in $\cp$,   with values in $(-\pi,\pi)$, is still continuous and the formula
$$
 \mathrm{arg\,}(fg) = \mathrm{arg\,}f + \mathrm{arg\,}g
$$
holds throughout $\D$. We will use this fact repeatedly.
\begin{thm} \label{thm-main}
Let $\f$ be a Schwarz-type function, $F\in\cp$, and denote by $\,\om$ the Schwarz-type function for which $F=\ell\circ\om$. Consider the argument function defined as above. Then the following conditions are equivalent:
\begin{itemize}
 \item[(a)]
$T_{F,\f}(\cp)\subset\cp$.
 \item[(b)]
$T_{F,\f}(\cl)\subset\cp$.
 \item[(c)]
The inequality
\begin{equation}
 4 |\f(z)|\cdot |\mathrm{Im\,}\om(z)| < (1-|\om(z)|^2) (1-|\f(z)|^2)
 \label{key-modif}
\end{equation}
holds for all $z$ in $\D$. In other words,
\begin{equation}
 2 |\f(z)|\cdot \left|\frac{\mathrm{Im\,}F(z)}{\mathrm{Re\,}F(z)} \right| < 1-|\f(z)|^2\,, \quad \mathrm{for \ all \ } z\in\D\,.
 \label{key-F}
\end{equation}
 \item[(d)]
The inequality
\begin{equation}
 |\arg F(z)| < \frac{\pi}2 - \arcsin \frac{2 |\f(z)|}{1+|\f(z)|^2}
 \label{ineq-arg}
\end{equation}
holds for all $z$ in $\D$. Note also that
\begin{equation}
 \frac{\pi}2 - \arcsin \frac{2 |\f(z)|}{1+|\f(z)|^2} = \frac{\pi}2 - \arctan \frac{2 |\f(z)|}{1-|\f(z)|^2} = \arctan \frac{1 - |\f(z)|^2}{2 |\f(z)|}\,,
 \label{eq-trig}
\end{equation}
where in the case when $z=0$ (recalling that $\f(0)=0$) the last equality should be understood as the limit:  $\arctan(+\infty)=\frac{\pi}{2}$.
\end{itemize}
\end{thm}
\par
It should be noted that in the above result the inequalities in conditions (c) and (d) are both invariant under rotations of $\f$ but not under the rotations in $\om$ (or under the appropriate changes in $F$).
\begin{proof}
We will show that (a) $\Leftrightarrow$ (b), (b) $\Leftrightarrow$ (c), and (c) $\Leftrightarrow$ (d).
\par\medskip
\fbox{(a) $\Leftrightarrow$ (b)}\,. The implication (a) $\Rightarrow$ (b) is obvious so we only have to see that (b) $\Rightarrow$ (a). First of all, the image of a convex combination of functions in $\cp$ under $T_{F,\f}$ is the same convex combination of their images and a convex combination of function in $\cp$ remains in $\cp$. Hence, if $T_{F,\f}(\cl)\subset\cp$ we also have that $T_{F,\f}(\textrm{co\,}(\cl)) = \textrm{co\,}T_{F,\f}(\cl) \subset\cp$ as $T_{F,\f}(\cl)\subset\cp$ by assumption and the class $\cp$ is clearly convex.
\par
Next, if $f_n\to f$ uniformly on compact subsets of $\D$, then also $F (f_n\circ\f) \to F (f\circ\f)$ in the same topology. Since $\cp$ is a compact family (in the classical terminology, meaning a closed set in the compact-open topology), we get $T_{F,\f}(\cp) = T_{F,\f}(\overline{\textrm{co\,}(\cl)}) = \overline{T_{F,\f} (\textrm{co\,}(\cl))} \subset\cp$.
\par\medskip
\fbox{(b) $\Leftrightarrow$ (c)}\,. To verify that \eqref{key-modif} is equivalent to \eqref{key-F}, one easily checks that if $F=\ell\circ\om$ then
$$
 \left|\frac{\mathrm{Im\,}F(z)}{\mathrm{Re\,}F(z)} \right| = 2 \frac{|\mathrm{Im\,}\om(z)|}{1-|\om(z)|^2}\,.
$$
\par
To see that (b) $\Rightarrow$ (c), suppose that $F (f\circ\f) \in \cp$ for all $f$ in $\cl$. In other words, $F (\ell_\la \circ\f) \in\cp$ for all $\la$ of modulus one and therefore also
$$
 F (\ell_\la \circ\f) = \frac{1+\om_\la}{1-\om_\la}
$$
for the Schwarz-type functions $\om_\la$ depending on each $\la$. This leads to the equation
$$
  \frac{1+\om}{1-\om}  \frac{1+\la \f}{1-\la \f} = \frac{1+\om_\la}{1-\om_\la}
$$
which holds in the entire unit disk. Solving for $\om_\la$, we get
$$
 \om_\la = \frac{\la\f+\om}{1+\la\f \om}\,.
$$
The condition $|\om_\la|<1$ in $\D$ is equivalent to
$$
 |\la\f+\om|^2<|1+\la\f \om|^2
$$
which amounts to the inequality
\begin{equation}
 |\f|^2 + |\om|^2 + 2 \mathrm{Re\,}\{\la\f \overline{\om}\} < 1 + |\f  \om|^2 + 2 \mathrm{Re\,}\{\la\f \om\}\,.
 \label{key-ineq}
\end{equation}
Grouping the terms in \eqref{key-ineq} we obtain
$$
 2 \mathrm{Re\,}\{\la \f(z) (\overline{\om (z)}-\om (z))\} < (1-|\om (z)|^2) (1-|\f(z)|^2)
$$
for each $z$ in $\D$ and for arbitrary $\la$ with $|\la|=1$. For each  point $z$ we can choose the argument of $\la$ appropriately so as to get
$$
 2 \mathrm{Re\,}\{\la \f(z) (\overline{\om (z)}-\om (z))\} = 4 |\f(z)|\cdot |\mathrm{Im\,}\om(z)|\,.
$$
Since this is valid at every point $z$ in the disk, the statement \eqref{key-modif} follows.
\par\medskip
To see that (c) $\Rightarrow$ (b), it suffices to observe that
$$
 2 \mathrm{Re\,}\{\la \f(z) (\overline{\om (z)}-\om (z))\} \le 4 |\f(z)|\cdot |\mathrm{Im\,}\om(z)|
$$
and it is now easy to reverse the steps in the above proof.
\par\medskip
\fbox{(c) $\Leftrightarrow$ (d)}\,. Since $F\in\cp$, we know that $|\arg F|<\pi/2$. Thus, inequality \eqref{key-F} is clearly equivalent to
$$
 2 |\f(z)|\cdot |\tan (\arg F(z))| < (1-|\f(z)|^2)\,, \qquad z\in\D\,,
$$
which is the same as
$$
 |\tan (\arg F(z))| < \frac{1-|\f(z)|^2}{2 |\f(z)|}\,, \qquad z\in\D\,,
$$
understanding the right-hand side as $+\infty$ when $\f(z)=0$. The inverse tangent function is odd so this is the same as
$$
 |\arg F(z)| < \arctan \frac{1-|\f(z)|^2}{2 |\f(z)|}\,, \qquad z\in\D\,.
$$
Equalities \eqref{eq-trig} follow by elementary trigonometry, so the proof is complete.
\end{proof}

\section{Some consequences and discussions}
 \label{sect-consequences}
\par
It should be stressed out that, even though our Theorem~\ref{thm-main} gives different characterizations of all admissible pairs of symbols, in some special situations the information given by the theorem can be made more precise. Actually, in some situations it may not be obvious  how many examples of admissible pairs can exist. The aim of this section is to explain what our main results amounts to in some important special situations.
\subsection{Some rigidity principles}
 \label{subsec-inner-comp}
\par
As is usual, for a bounded analytic function $\f$ in $\D$ we write $\|\f\|_\infty =\sup_{z\in\D}|\f(z)|=\esssup_{\z\in\T}|\f(\z)|$.
\par
Recall that bounded analytic functions in $\D$ have radial limits $\f(\z)=\lim_{r\to 1} \f(r \z)$ for almost every point $\z$ on the unit circle $\T$ with respect to the normalized Lebesgue arc length measure \cite[Chapter~1]{D1}. An analytic function $\f$ in the disk is called \textit{inner\/} if $|\f(z)|\le 1$ for all $z$ in $\D$ (equivalently, $\|\f\|_\infty\le 1$) and also $|\f (\z)|=1$ almost everywhere on $\T$. The following result generalizes our Proposition~\ref{prop-no-mult}.
\begin{prop} \label{prop-fi-inner}
Let $F\in\cp$ and let $\f$ be inner. Then $T_{F,\f}(\cp) \subset\cp$ if and only if \,$F\equiv 1$.
\end{prop}
\begin{proof}
The bounded functions $\f$ and $\om$ have radial limits almost everywhere on the circle. Thus, for almost every $\z\in\T$ we may pass to the limit as $z\to\z$ in inequality \eqref{key-modif} to conclude that $\mathrm{Im\,}\om (\z)=0$ almost everywhere on $\T$. Now it is an easy exercise to see that this together with $\om (0)=0$ implies $\om=0$. Just consider the bounded analytic function $g=\exp\{i\om\}$ in $\D$ whose boundary values on the circle have modulus one almost everywhere, hence $\|g\|_\infty=1$, and note that $g(0)=1$; it follows that $g\equiv 1$, hence $\om\equiv 0$ (that is, $F\equiv 1$).
\end{proof}
Here is the counterpart of this statement with assumptions on $\om$.
\begin{prop} \label{prop-om-inner}
Let $F=\ell\circ\o$, where $\om$ is an inner function. Then $T_{F,\f}(\cp) \subset\cp$ if and only if $\f\equiv 0$.
\end{prop}
\begin{proof}
After passing to the radial limits in \eqref{key-modif} we get that
$\f  \mathrm{Im\,}\om=0$ almost everywhere on the unit circle.
\par
If $\f=0$ only on a set of measure zero on the circle, then Im\,$\om=0$ almost everywhere on the circle. From the proof of the previous theorem we know that $\om\equiv 0$, which contradicts our initial assumption.
Hence $\f=0$ on a set of positive measure. By a classical theorem of Nevanlinna \cite[Theorem~2.2]{D1}, it follows that $\f\equiv 0$.
\end{proof}
It is easily seen from Theorem~\ref{thm-main} that any admissible multiplication symbol $F$ can only carry a very small portion of the boundary of the unit disk to the imaginary axis.
\begin{prop} \label{prop-meas-0}
Let $F\in\cp$, let $\f$ and $\om$ be two Schwarz-type functions,  $\f\not\equiv 0$, and suppose that $F=\ell\circ\om$ and $T_{F,\f}$ preserves \,$\cp$ as before. Denote the radial limits of $F$ again by $F$ and let
$$
 A = \{\z\in\T\,\colon\,\mathrm{Re\,}F(\z)=0\} = \{\z\in\T\,\colon\,|\om(\z)|=1, \,\om(\z)\neq 1\}\,.
$$
Then $m(A)=0$.
\end{prop}
\begin{proof}
Assume the contrary: $m(A)>0$. After passing on to the radial limits in \eqref{key-modif}, we obtain
$$
 4 |\f(\z)|\cdot |\mathrm{Im\,}\om(\z)| \le (1-|\om(\z)|^2) (1-|\f(\z)|^2)
$$
for almost all $\z$ with $|\z|=1$. Specifically, $\,\f \mathrm{Im\,}\om =0$ holds at almost every point of $A$ (note that $\f$ may not have radial limits at some subset of $A$ of total measure zero). Since the measure of $A$ is positive and $\f\not\equiv 0$, we must have either $\om=1$ or $\om=-1$ on a set of positive measure in $A$. The first case is excluded by the definition of $A$ and in the second case the Nevanlinna theorem implies that $\om\equiv -1$ in $\D$, which is impossible in view of the assumption that $\om(0)=0$. This shows that $m(A)=0$.
\end{proof}
\par\medskip
In the context of (linear) weighted composition transformations the case in which $F=\f^\prime$ is often important. However, in our context it should be noted that in this case we only obtain another rigidity situation. Namely, assuming that $T_{\f^\prime,\f}(\cp) \subset\cp$ and choosing $f\equiv 1$ we get $\phi^\prime\in\cp$ hence $\f^\prime (0)=1$. The case of equality in the Schwarz lemma forces $\f(z)=z$, hence $F=\f^\prime\equiv 1$, so our transformation $T_{\f^\prime,\f}$ reduces to the identity map.

\subsection{Cases where one of the symbols has small range}
 \label{subsec-comp-cont}
Many non-trivial examples of weighted composition transformations that preserve class $\cp$ are possible when $\f(\D)$ is compactly contained in $\D$ or $F(\D)$ is contained in a sector, as the following results show.
\begin{prop} \label{prop-comp-cont}
Let $F\in\cp$ and let $\f$ be a Schwarz-type function such that $\|\f\|_\infty=R<1$. Then whenever the function $F$ satisfies
$$
  |\arg F(z)| < \frac{\pi}{2} - \arcsin \frac{2 R}{1+R^2}
$$
for all $z$ in $\D$, we have that $T_{F,\f}(\cp)\subset\cp$.
\end{prop}
\begin{proof}
Follows from criterion (d) of Theorem~\ref{thm-main} and the fact that the function $2x/(1+x^2)$ is increasing in the interval $(0,1)$.
\end{proof}
\par\medskip\noindent
\textbf{Example 1}. An explicit example is $\f(z)=R z$, $0<R<1$, and
$$
 F(z)= \( \frac{1+z}{1-z} \)^\e\,, \qquad 0 < \e < 1 - \frac{2}{\pi} \arcsin \frac{2 R}{1+R^2}\,,
$$
a conformal map of the unit disk onto an angular sector with vertex at the origin. Condition \eqref{ineq-arg} is clearly satisfied.
\par
\begin{prop} \label{prop-comp-cont-dual}
Let $F\in\cp$ and $K=\sup_{z\in\D} |\arg F(z)|<\frac{\pi}{2}$. Write $K=\arcsin\frac{2R}{1+R^2}$, $0\le R<1$. If \,$\f$ is a Schwarz-type  function such that
$$
  \|\f\|_\infty \leq \frac{1-R}{1+R}
$$
then $T_{F,\f}(\cp)\subset\cp$.
\end{prop}
\begin{proof}
By assumption,
$$
 |\arg F(z)| \le K = \arcsin\frac{2R}{1+R^2}\,, \qquad z\in\D\,.
$$
In view of condition \eqref{ineq-arg} from Theorem~\ref{thm-main} it suffices to check that
$$
 \arcsin\frac{2R}{1+R^2} < \frac{\pi}{2} - \arcsin\frac{2 |\f(z)|}{1+|\f(z)|^2}
$$
holds for all $z$ in $\D$. Equivalently,
$$
  \arcsin\frac{2 |\f(z)|}{1+|\f(z)|^2} < \frac{\pi}{2} - \arcsin\frac{2R}{1+R^2}
$$
must hold throughout $\D$. This will certainly be satisfied if
\begin{equation}
  \arcsin\frac{2 \|\f\|_\infty}{1+\|\f\|_\infty^2} \le \frac{\pi}{2} - \arcsin\frac{2R}{1+R^2}
  \label{arc-sin}
\end{equation}
in view of monotonicity of the sine function in $[0,\frac{\pi}{2})$ and of $u(x)=\frac{2x}{1+x^2}$ in $[0,1)$. But \eqref{arc-sin} is clearly equivalent to
$$
 \frac{2 \|\f\|_\infty}{1+\|\f\|_\infty^2} \le \sin \(\frac{\pi}{2} - \arcsin\frac{2R}{1+R^2}\) = \cos \(\arcsin\frac{2R}{1+R^2}\) = \sqrt{1-\(\frac{2R}{1+R^2}\)^2} = \frac{1-R^2}{1+R^2}\,.
$$
This yields an elementary quadratic inequality in $\|\f\|_\infty$ which is easily seen to be satisfied whenever
$$
 0\le \|\f\|_\infty \le\frac{1-R}{1+R}\,.
$$
This proves the statement.
\end{proof}
\par\medskip
We now formulate a counterpart of Proposition~\ref{prop-comp-cont} with similar hypotheses on $\om$ instead of $\f$ which follows from our previous result.
\par
\begin{cor} \label{cor-comp-cont-dual}
Let $F=\ell\circ\om\in\cp$, where $\om$ is a Schwarz-type function. If $\,\|\om\|_\infty<1$ and $\f$ is a Schwarz-type function such that
$$
  \|\f\|_\infty < \frac{1-\|\om\|_\infty}{1+\|\om\|_\infty}
$$
then $T_{F,\f}(\cp)\subset\cp$.
\end{cor}
\begin{proof}
Let $R=\|\om\|_\infty<1$. Then the function $F$ is clearly subordinated to the function
$$
 \ell_R=\frac{1+Rz}{1-Rz}
$$
in the usual sense that $F=\ell_R\circ \(\frac{\om}{R}\)$. Thus, $F(\D) \subset \ell_R(\D)$. It is plain that $\ell_R(\D)$ is the disk whose diameter has endpoints
$$
 \ell_R (-1) = \frac{1-R}{1+R}\,, \qquad \ell_R (1) = \frac{1+R}{1-R}\,,
$$
hence its center and radius are respectively
$$
 C=\frac{1+R^2}{1-R^2}\,, \qquad \r=\frac{2R}{1-R^2}\,.
$$
Let us denote by $C_R=\{z\in\C\,\colon\,|z-C|=\r\}$ the boundary of this disk. Let $a$ be the point of intersection of the circle $C_R$ with its tangent from the origin in the upper half-plane. By looking at the right triangle determined by the origin and the points $a$ and $C$, we infer that
$$
  \arg a = \arcsin \frac{\r}{C} = \arcsin \frac{2 R}{1+R^2} \,.
$$
One argues similarly for the point of tangent in the lower half-plane and obtains that, for every $z$ in $\D$,
$$
 |\arg F(z)| < \arcsin \frac{2 R}{1+R^2}\,.
$$
The conclusion now follows from Proposition~\ref{prop-comp-cont-dual}.
\end{proof}
\par\medskip
Many interesting examples in geometric function theory are obtained from special types of conformal mappings such as the lens maps. In what follows, for $0<\a<1$ we will denote by $\la_\a$ the standard  \textit{lens map\/} given by the formula
$$
 \la_\a(z)=(\ell^{-1}\circ\ell^\a)(z) = \frac{\ell^\a(z)-1}{\ell^\a(z)+1} \,, \qquad z\in\D\,.
$$
It is elementary that $\la_\a$ is a conformal map of the unit disk onto a lens-shaped region $L_\a$ bounded by two circular arcs (symmetric with respect to the real axis) that intersect at the points $\pm 1$ forming an angle of opening $\pi\a$ at each of these points; see \cite[p.~27]{S}.
\par
The following simple geometric observation will be useful. The half-plane map $\ell$ is bijective between a lens-shaped region $L_\a$ and an angle with vertex at the origin and maps in a one-to-one fashion the largest disk contained in the lens-shaped region onto a disk tangent to the legs of the angle.
\par
Our next result essentially shows that when the multiplication symbol is obtained by composing the half-plane map with a lens map, the statements of Proposition~\ref{prop-comp-cont} and Proposition~\ref{prop-comp-cont-dual} can be unified into a single ``if and only if'' statement.
\par\medskip
\begin{prop}\label{prop-lens}
Let $F=\ell\circ\la_\a$, where $\la_\alpha$ is a lens map, and let $\phi$ be a Schwarz-type function. Then $T_{F,\phi}(\mathcal{P}) \subset\mathcal{P}$ if and only if
$$
 \|\phi\|_{\infty}\leq\frac{1-R}{1+R}\,,
$$
where $\frac{2R}{1+R^2}=\sin{\frac{\alpha\pi}{2}}$.
\end{prop}
\begin{proof}
Suppose first that $\|\phi\|_{\infty}\leq\frac{1-R}{1+R}$. Then, since $$
 \sup_{z\in\mathbb{D}}|\arg\,F(z)| = \frac{\alpha\pi}{2} =K= \arcsin\frac{2R}{1+R^2}\,,
$$
by Proposition~\ref{prop-comp-cont-dual} it follows that $T_{F,\phi} (\mathcal{P}) \subset \mathcal{P}$. Now, if $T_{F,\varphi}(\mathcal{P}) \subset\mathcal{P}$, by condition \eqref{ineq-arg} of Theorem~\ref{thm-main} we have
$$
 |\arg\,F(z)|+\arcsin \frac{2|\phi(z)|}{1+|\phi(z)|^2} <\frac{\pi}{2}\,.
$$
Therefore, for almost every $\zeta\in\mathbb{T}$,
$$
 \arcsin \frac{2|\phi(\zeta)|}{1+|\phi(\zeta)|^2}\leq \frac{\pi}{2}-\frac{\alpha\pi}{2}=\frac{\pi}{2}-\arcsin\frac{2R}{1+R^2}\,.
$$
Then, as in the proof of Proposition~\ref{prop-comp-cont-dual},
$$
 \frac{2\|\phi\|_\infty}{1+\|\phi\|_\infty^2}\leq\sin\left(\frac{\pi}{2} -\arcsin\frac{2R}{1+R^2}\right)=\frac{1-R^2}{1+R^2},
$$
and from here,
$$
 \|\phi\|_{\infty}\leq\frac{1-R}{1+R}\,.
$$
\end{proof}

\subsection{Composition symbols with radial limits of modulus one and/or angular derivatives}
 \label{subsec-ang-deriv}
Recall that an analytic self-map $\f$ of $\D $ is said to have an
\emph{angular derivative\/} $\f^\prime(\z)$ (in the restricted
sense of Carath\'eodory \cite[\S~298-299]{C}) at a point $\z $ on the unit circle $\T $
if it satisfies the following two conditions:
\par
(a) the nontangential limit of $\f$ at $\z $ has modulus one;
\par
(b) $\f^\prime(z)$ has a finite nontangential limit as $z\to\z $.
\par\smallskip
The Julia-Carath\'eodory theorem (see \cite{CM} or \cite{S}) states that $\f$ has an angular derivative at $\z $ if and only if
$$
 \liminf_{z\to\z}\frac{1-|\f(z)|}{1-|z|}<\infty
$$
and, in this case, $|\f^\prime(\z)|$ equals the above
(unrestricted) lower limit (note that this limit is always strictly positive \cite[p.~57]{S}. Otherwise it is understood that
$|\f^\prime(\z)|=\infty $.
\par
Even though the angular derivative of $\f$ need not exist
anywhere on $\T $ as a finite number, the function $|\f^\prime| :
\T\to [0,\infty] $ is well defined in this extended sense; being
lower semicontinuous (\cite{BS}, Lemma~2.5), it attains its
minimum on $\T $ (\emph{cf.\/} also \cite[Proposition~2.46]{CM}).
\par
One ought to keep in mind that the function $|\f^\prime|$ on $\T
$ as above, in general, does not coincide at all with the modulus
of the boundary values of $\f^\prime $ (if those exist). The most
obvious example is the linear map $\f(z)=az+b $ onto a disk
compactly contained in $\D $, which happens precisely when $|a|+|b|<1 $. Its usual derivative is constant everywhere, while the angular derivative does not exist at any point on the boundary; in this case, we interpret that $|\f^\prime(\z)|=\infty$ for every point $\z$ on the unit circle.
\par
The concept of angular derivative is fundamental in the study of compactness of composition operators on Hardy and Bergman spaces, as well as in the iteration of analytic self-maps of the unit disk.
\par\medskip
Our next result shows that if $\f$ possesses even a mildly reasonable boundary behavior at a point on the unit circle then $\om$ automatically cannot be ``too good'' at the same point.
\begin{thm} \label{thm-ang-deriv}
Let $F$, $\om$, and $\f$ be as before and suppose that at some point $\z$ on the unit circle the function $\f$ has radial limit of modulus one. Then if the transformation $T_{F,\f}$ preserves $\cp$, the function $\om$ cannot have a finite non-zero angular derivative at $\z$.
\end{thm}
\begin{proof}
Note that $T_{F,\f}$ preserves $\cp$ if and only if the transformation $T_{F_\la,\f_\la}$ preserves $\cp$, where $F_\la (z)=F(\la z)$ and $\f_\la (z)=\f (\la z)$, whenever $|\la|=1$. Hence, we may assume without loss of generality that $\zeta=1$.
\par
Suppose that $\om$ has a finite angular derivative at $\z=1$. Then the radial limit $\om(1)$ exists and $|\om(1)|=1$. Taking the angular limit as $z\to 1$ in \eqref{key-modif}, we conclude that Im\,$\{\om(1)\}=0$. Thus, either $\om(1)=-1$ or $\om(1)=1$.
\par
Let us first consider the case $\om(1)=-1$. Since at $\z=1$ the angular derivative of $\om$ is neither $0$ nor $\infty$, we know \cite[p. 291]{P} that it is actually univalent in some Stolz domain with vertex at $z=1$:
\[
 \Delta=\left\{z\colon |\arg\,  (1-z)|< \t, \ r<|z|<1\right\}
\]
for suitable $r\in (0,1)$ and $\t>0$. Also, as is well known (\textit{cf.} again \cite[p. 291]{P}), the function $\om$ preserves angles between curves contained in $\De\cup \{1\}$ that meet at $z=1$. This shows that there exists a curve $\g:[0,1]\to\Delta\cup\{1\}$ with $\g(1)=1$ and which is mapped by $\om$ onto some non-horizontal segment
$$
 S=\{-1+se^{i \a_0}\colon 0\le s\le s_0\}\,, \quad 0<|\a_0|<\frac{\pi}{2}\,,
$$
for an appropriate value of $s_0$. (To see this, it suffices to look at the image under $\om$ of the suitable Stolz domain mentioned earlier  with vertex at $1$, which will contain another Stolz domain with vertex at $-1$, and to select $\a_0$ and $s_0$ so that the segment $S$ is contained in this new Stolz angle and is not contained in the real axis.) Keeping in mind that $F=\ell\circ\om$ and $\ell$ is a M\"obius transformation which maps the diameter $(-1,1)$ to the positive semi-axis, we see that
\[
 \arg\,  F(\g(t)) = \arg\,   \ell(\om(\g(t))) \to \a_0\,,\quad\text{as}\quad t\to 1^-\,.
\]
Therefore, taking the limit as $z\to 1$ along $\g$ in \eqref{ineq-arg}, we obtain $|\a_0|\le 0$, which is contrary to our construction of the segment $S$. This completes the proof in the case when $\om(1)=-1$.
\par
By \eqref{key-modif}, $T_{F,\f}$ preserves $\cp$ if and only if $T_{G,\f}$ with $G=(1-\om)/(1+\om)$ does, so we can argue as above in the case when $\om(1)=1$ to get a contradiction again.
\end{proof}
\par\medskip
There are two ways in which the function $\om$ can fail to have angular derivative: either it does not have a radial limit of modulus one or it does but the differential quotient fails to have a limit at the point in question. Here is an example of the first kind. It deals with the map $\f$ such that $\f(\D)$ has a tangential contact with the unit circle. The price we pay for this is that $\om$ is a dilated self-map of the disk (hence, in this example $\om(\D)$ is compactly contained in $\D$).
\par\medskip\noindent
\textbf{Example 2}. For $K\ge 3/2$, let
$$
 \f(z)=\frac{z(1+z)}{2}\,, \qquad \om(z)=\frac{z(2-z)}{2K}\,.
$$
Both are clearly Schwarz-type functions. Obviously, $\f(1)=1$ and $\f$ is conformal at $z=1$ since $\f^\prime(1)\neq 0$. For a sufficiently large value of $K$ (which will be determined below) one can also check that our condition \eqref{key-modif} is satisfied, hence $T_{F,\f}(\cp) \subset\cp$. Indeed, it is immediate that
$$
 \mathrm{Im\,}\om(z)=\frac{y(1-x)}{K}\,, \qquad z=x+iy\,.
$$
Checking our condition \eqref{key-modif} in this case reduces to verifying that
$$
 \frac{2\,\left| z (1+z) y \right|\,(1-x)}{K} < \(1 - \left| \frac{z(1+z)}{2}\right|^2\) \(1 - \left| \frac{z(2-z)}{2 K}\right|^2\)
$$
holds for all $z$ in $\D$. (Note that as $z\to 1$, both sides tend to zero but the strict inequality is maintained.) Since $x^2+y^2=|z|^2<1$, it is clear that
$$
 \frac{2\,\left| z (1+z) y \right|\,(1-x)}{K} < \frac{4 (1-x)}{K}
$$
while the right-hand side can be estimated from below as follows:
\begin{eqnarray*}
 \(1 - \left| \frac{z(1+z)}{2}\right|^2\) \(1 - \left| \frac{z(2-z)}{2 K}\right|^2\) &>& \(1 - \frac{|1+z|^2}{4}\) \(1 - \frac{9}{4 K^2}\)
\\
 & = & \(1 - \frac{(1+x)^2 + y^2}{4}\) \(1 - \frac{9}{4 K^2}\)
\\
 &\ge & \(1 - \frac{2+2x}{4}\) \(1 - \frac{9}{4 K^2}\)
\\
 &=& \frac{1-x}{2} \(1 - \frac{9}{4 K^2}\)
\end{eqnarray*}
so it is only left to check that
$$
 \frac{4 (1-x)}{K} < \frac{1-x}{2} \(1 - \frac{9}{4 K^2}\)
$$
for $K$ large enough and $|x|<1$, which is clear. The inequality holds for all $K>4+\dfrac{\sqrt{73}}{2}$.
\par\medskip
The natural question arises as to whether it is possible to have an example where both $\f$ and $\om$ can have radial limits of modulus one at the same point (obviously, without having an angular derivative at the point in question) but the weighted composition $T_{F,\f}$ still preserves $\cp$. The following example, illustrated by the figure below, gives an affirmative answer.
\par\medskip\noindent
\textbf{Example 3}. Consider the planar domain
$$
 \O = \{x+iy\,\colon\, 4 |y| \sqrt{x^2+y^2} < (1-x^2-y^2)^2\}\,,
$$
clearly symmetric with respect to both the real and imaginary axes. Let $\om$ be a conformal map of $\,\D$ onto $\O$ which fixes the origin. Starting with the subdomain of $\O$ in the upper half-plane and using the Schwarz reflection principle, one can also choose $\om$ in such a fashion that it fixes the diameter $(-1,1)$ and $\om(1)=1$ in the sense of a radial limit. Let $\f=\om$. It can now easily be checked that our condition \eqref{key-modif} is satisfied, hence $T_{F,\f}(\cp) \subset \cp$.
\par
Note, however, in relation to this ``leaf-shaped'' region that our mapping $\om=\f$ has boundary contact with the unit circle but does not have angular derivative at $z=1$. The intuitive reason for this is that the corners at $-1$ and $1$ are contained in lens-shaped regions and lens maps do not have angular derivatives. A rigorous proof of this fact can be given by using subordination. Alternatively, one can write  the equation of the boundary of $\O$ in polar coordinates:
$$
 4 r^2 |\sin \t| = (1-r^2)^2\,.
$$
Solving for $r$, one obtains
$$
 1-r^2 = 2 \(\sqrt{\sin |\t|+\sin^2 |\t|}-\sin |\t|\)\,,
$$
and by elementary calculus one checks that for a fixed $\e>0$ the integral
$$
 \int_0^\e \frac{1-r(\t)}{\t^2}\,d\t
$$
diverges and the conclusion follows by the Tsuji-Warschawski criterion \cite[p.~72]{S}.
\par\medskip

\begin{figure}[htbp]
{\centering
\includegraphics[width=\textwidth]{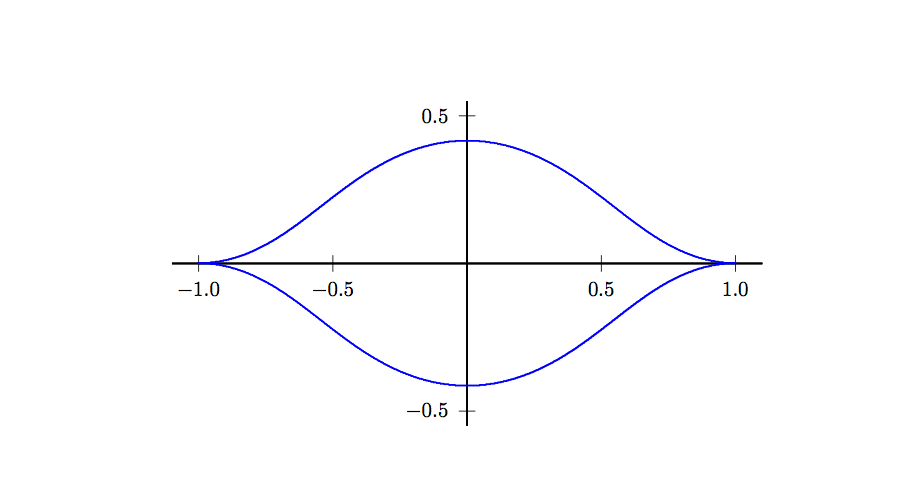}
\caption{The boundary of the leaf-shaped region $\O$.}\par\medskip
}
\end{figure}

\par\medskip
\section{Fixed points of weighted composition transformations that preserve $\cp$}
 \label{sect-fixed-pts}
\par
Even though we are working in a non-linear context, it is possible to adapt the arguments typical for such situations; \textit{cf.} \cite[Sect.~6.1]{S}.
\par
\begin{thm} \label{thm-fixed-pts}
Let $T_{F,\f}$ be a weighted composition transformation such that $T_{F,\f}(\mathcal{P})\subset\mathcal{P},$ where $F=\ell\circ \om,$ $\f$ and $\om$ are Schwarz-type functions, and $\f$ is not a rotation. Then $T_{F,\f}$ has a unique fixed point which is obtained by iterating $T_{F,\f}$ applied to arbitrary $f$ in $\cp$.
\par
In the case when $\,\f$ is inner but not a rotation, the unique fixed point is the constant function one.
\end{thm}
\begin{proof}
We first show that the limit of iterates of $T_{F,\f}$ applied to an arbitrary function $f$ in $\cp$ is a fixed point of the transformation.  Define the iterations of $\f$ in the usual way, $\f_0$ being the identity function and $\f_{n+1}=\f_n\circ\f$, $n\ge 0$. Let $f\in\cp$. It is easy to see by induction that
$$
 F (F\circ\f)\ldots (F\circ\f_{n-1}) (f\circ\f_n) = T^n_{F,\f}f \in\mathcal{P}
$$
for any integer $n\ge 1$. By our assumptions on $\f$, the origin is its only fixed point in $\D$. Since $\f$ is not a disk automorphism, it follows that $\f_n \to 0$ uniformly on compact subsets of $\D$ and therefore also $\om \circ\f_n\to 0$ in the compact-open topology  as $n\to\infty$. Thus,  $f\circ\f_n\to 1$ uniformly on compact subsets as $n\to\infty$. On the other hand,
$$
 \prod_{k=0}^{n-1} (F\circ\f_k) = \prod_{k=0}^{n-1}\frac{1+\om  \circ\f_k}{1-\om \circ\f_k}\,,
$$
so proving the uniform convergence on compact subsets of the infinite product $\prod_{k=0}^{\infty} (F\circ\f_k)$ is equivalent to proving the convergence on compact subsets of $\D$ of the sums
$$
 \sum_{k=0}^{n-1}\left|1-\frac{1+\om \circ\f_k}{1-\om \circ\f_k}\right| = 2\sum_{k=0}^{n-1}\left|\frac{\om \circ\f_k}{1-\om \circ\f_k}\right|\,.
$$
For $r\in (0,1)$ fixed, let $m(r)=\max_{|z|\le r}|\f(z)|$. Let $\delta=m(r)/r$. Clearly, $\d<1$ since $\f$ is not a rotation. Applying the Schwarz lemma to $\f(r w)/m(r)$, we get $|\f(r w)|/m(r)\le |w|$ for any $w\in\mathbb{D}$, and from here
$$
 |\f(z)|\leq\frac{m(r)}{r}|z|=\delta|z|
$$
whenever $|z|\leq r$. Iterating this inequality, we get
$$
 |\f_k(z)|\le \delta|\f_{k-1}(z)|\leq\ldots\leq\delta^k|z|
$$
for $|z|\le r$. Using the fact that $\om$ is a Schwarz-type function, we obtain
$$
 \left|1-\frac{1+\om \circ\f_k}{1-\om \circ\f_k}\right| = 2 \left|\frac{\om \circ\f_k}{1-\om \circ\f_k}\right|\le 2 \frac{|\om \circ\f_k|}{1-|\om \circ\f_k|} \le 2 \frac{|\f_k|}{1-|\f_k|}
\le 2 \frac{\delta^k r}{1-\delta^k r} \le \frac{2 r}{1-r} \delta^k
$$
in the disk $\{z\,\colon\,|z|\le r\}$. Thus, the series
$$
 \sum_{k=0}^\infty \left|1-\frac{1+\om \circ\f_k}{1-\om \circ\f_k}\right|
$$
converges uniformly on compact subsets of the disk and the infinite product $\prod_{k=0}^{\infty} (F\circ\f_k)$ is uniformly convergent on compact subsets to some function $G$ analytic in $\D$. Moreover, since $\mathcal{P}$ is a compact class, $G\in\mathcal{P}$. Combining both limits, we obtain
$$
 T^n_{F,\f}f = F (F\circ\f) \ldots (F\circ\f_{n-1}) (f\circ\f_n) \to G
$$
uniformly on compact subsets as $n\to\infty$ for any $f\in\mathcal{P}$. Now we can see that $G$ is a fixed point of the transformation. Applying $T_{F,\f}$ to $G$ we have
\begin{align*}
 T_{F,\f}G & = F (G\circ\f) = F \left(\lim_{n\to\infty}\prod_{k=0}^{n-1} (F\circ\f_k)\right) \circ\f = F\left(\lim_{n\to\infty}\prod_{k=0}^{n-1} (F\circ\f_{k+1}) \right)
\\
 &=F\left(\lim_{n\to\infty}\prod_{k=1}^{n} (F\circ\f_{k})\right)= \lim_{n\to\infty}\prod_{k=0}^{n} (F\circ\f_{k})=G\,.
\end{align*}
Note that $G$ as constructed above does not depend on the initial choice of the function $f$ in $\cp$. Now it is clear that this $G$ is the only fixed point of $T_{F,\f}$, because if $g\in\mathcal{P}$ satisfies $T_{F,\f}g=g$, iterating the transformation we get
$$
 g=T^n_{F,\f}g\to G
$$
uniformly on compact subsets of $\D$.
\par
It is only left to check our final comment in the statement of the theorem. Since $\f$ is an inner function, by Proposition~\ref{prop-fi-inner} we have $F\equiv 1$ so the equation for the fixed point: $f\circ\f=f$ is the classical Shr\"oder equation for the composition operator $C_\f$ corresponding to the eigenvalue $\la=1$. By the first proposition from \cite[Section~6.1]{S}, if $f$ were not identically constant, it would follow that $|\f^\prime(0)|=1$. In view of the Schwarz lemma, this forces $\f$ to be a rotation, contrary to our assumptions.
\end{proof}
\par
The case when $\f$ is a rotation leads to the well-known case of fixed points from the theory of composition operators, describing a trichotomy: the identity map, a rational rotation or an irrational rotation.
\begin{prop} \label{prop-fixed-pts-rot}
Let $T_{F,\f}$ be a weighted composition transformation such that $T_{F,\f}(\mathcal{P})\subset\mathcal{P}$, where $F=\ell\circ \om$, $\om$ is a Schwarz-type functions, and $\f$ is a rotation. Then the set of all fixed points of $T_{F,\f}$ is as follows:
\par
(a) all of $\,\cp$, if $\f(z)\equiv z$;
\par
(b) the functions with $n$-fold symmetry: $f(z)=g(z^n)$, $g\in\cp$, whenever $\f(z)=\la z$, where $\la^n=1$ for some $n>1$;
\par
(c) only the constant function one, if $\f(z)=\la z$, where $|\la|=1$ and $\la^n\neq 1$ for all $n\in\N$.
\end{prop}
\begin{proof}
Since $\f$ is an inner function, Proposition~\ref{prop-fi-inner} forces $F\equiv 1$, hence $T_{F,\f}f=f\circ\f$. Part (a) now follows trivially.
\par
Parts (b) and (c) follow readily by comparing the Taylor series of both sides of the equality $f(\la z)=f(z)$ in the disk.
\end{proof}


\end{document}